\documentclass[times,sort&compress,3p]{elsarticle}
\journal{}
\usepackage[labelfont=bf]{caption}

\usepackage{amsmath,amsfonts,amssymb,amsthm,booktabs,color,epsfig,graphicx,hyperref,url}

\usepackage{natbib}
\usepackage{bbm}
\usepackage{textcomp}

\usepackage{framed}
\usepackage{comment}
\definecolor{shadecolor}{gray}{0.9}

\specialcomment{extra}{\begin{shaded}}{\end{shaded}}

\theoremstyle{plain}  %default 
\newtheorem{theorem}{Theorem}
\newtheorem{lemma}[theorem]{Lemma} 
\newtheorem{proposition}[theorem]{Proposition} 
\newtheorem{corollary}[theorem]{Corollary} 

\theoremstyle{definition} 
\newtheorem{defn}{Definition} 

\newtheorem{example}{Example}

\theoremstyle{remark}

\newcommand{\prob}{\mathbb{P}}

\newcommand{\diff}{\,\mathrm{d}}
\newcommand{\E}{\mathbb{E}}

\newcommand{\one}{\mathbbm{1}}
\newcommand{\R}{\mathbb{R}}

\newcommand{\eps}{\varepsilon}
\renewcommand{\emptyset}{\varnothing}

\DeclareMathOperator{\ent}{ent}

\begin{document}

\begin{frontmatter}

\title{Bivariate distributions with ordered marginals}

\author[A1]{Sebastian Arnold}
\author[A1]{Ilya Molchanov}
\author[A1]{Johanna F.~Ziegel\corref{mycorrespondingauthor}}

\address[A1]{Institute of Mathematical Statistics and Actuarial Sciences, University of Bern, Switzerland}

\cortext[mycorrespondingauthor]{Corresponding author. Email address: \url{johanna.ziegel@stat.unibe.ch}}

\begin{abstract}
  This paper provides a characterization of all possible dependency
  structures between two stochastically ordered random variables. The
  answer is given in terms of copulas that are compatible
  with the stochastic order and the marginal distributions. The extremal values for Kendall's $\tau$
  and Spearman's $\rho$ for all these copulas are given in closed
  form. We also find an explicit form for the joint distribution with
  the maximal entropy. A multivariate extension and a generalization
  to random elements in partially ordered spaces are also provided.
\end{abstract}


\end{frontmatter}

\section{Introduction}

Let $X_1$ and $X_2$ be two random variables, such that $X_1$ is
\emph{stochastically larger} than $X_2$. This means that $F_1(x) \le
F_2(x)$, $x \in \R$, for their cumulative distribution functions
(cdfs) $F_1$ and $F_2$, respectively. It is well known that this is
the case if and only if $X_1$ and $X_2$ can be realized on the same
probability space, so that $X_1\geq X_2$ almost surely. The objective
of this paper is to characterize all random vectors $(X_1, X_2)$ such
that $X_1$ and $X_2$ have given cdfs and $\prob(X_1\geq X_2)=1$.

As a first observation, we establish a representation of joint
distributions of ordered random variables as distributions of the
order statistics sampled from an exchangeable bivariate law. 

\begin{theorem}
  \label{thr:order}
  A random vector $(X_1,X_2)$ with marginal cdfs $F_1$ and $F_2$ satisfies $\prob(X_1\geq X_2)=1$ if and
  only if $X_1=\max(V_1,V_2)$ and $X_2=\min(V_1,V_2)$ for a random
  vector $(V_1,V_2)$ with exchangeable components and such that
  $\prob(V_1\leq x,V_2\leq x)=F_1(x)$ and 
  $\prob(V_i\leq x)=G(x)$ for all $x$, $i\in\{1,2\}$, where
  \begin{equation}
    \label{eq:2}
    G(x):=\frac{1}{2}(F_1(x)+F_2(x)), \quad x\in\R.
  \end{equation}
\end{theorem}
\begin{proof}[\textnormal{\textbf{Proof}}]
  The vector $(V_1,V_2)$ obtained as the random permutation of
  $(X_1,X_2)$ is exchangeable and its marginal distributions are
  $G$. Furthermore,
  \begin{displaymath}
    \prob(V_1\leq x,V_2\leq x)
    =\prob(X_1\leq x,X_2\leq x)=\prob(X_1\leq x).
  \end{displaymath}

  Conversely, if $X_1$ and $X_2$ are order statistics from
  $(V_1,V_2)$, then $X_1\geq X_2$ a.s., and 
  \begin{align*}
    \prob(X_1\leq x)&=\prob(V_1\leq x,V_2\leq x)=F_1(x),\\
    \prob(X_2\leq x)&=\prob(V_1\leq x,V_2\leq x)
    +\prob(V_1\leq x,V_2>x)+\prob(V_1>x,V_2\leq x)\\
    &=2 G(x)-F_1(x)=F_2(x).  \qedhere
  \end{align*}
\end{proof}

This theorem complements already known results deriving the
distribution of order statistics from general multivariate laws, see
\cite{nav:spiz10,leb:dut14,diet:fuc:sch16}.

If the supports of $X_1$ and $X_2$ are disjoint intervals, then any dependency
structure between them is possible. Otherwise, restrictions are
necessary, e.g.\ $X_1$ and $X_2$ cannot be independent. In
Section~\ref{sec:char-stoch-order}, we give a complete description of
the joint distribution of $X_1$ and $X_2$.
This description is given in terms of copulas and their diagonal
sections. In Section~\ref{sec:maxim-indep-pairs}, we identify bounds on the joint
distribution of $(X_1,X_2)$ in terms of concordance ordering. Then in
Section~\ref{sec:nonp-corr-coeff}, we determine the smallest possible
nonparametric correlation coefficients. The joint distribution of
$(X_1,X_2)$ with the maximal entropy is found in
Section~\ref{sec:maxim-entr-solut}, followed by examples in
Section~\ref{sec:examples}. A multivariate extension and a generalization to random elements in
partially ordered spaces are presented in
Section~\ref{sec:generalizations}.

\section{Characterization of stochastically ordered copulas}
\label{sec:char-stoch-order}

A (bivariate) copula $C:[0,1]^2 \to [0,1]$ is the cdf of a
random vector $(U_1,U_2)$ with standard uniformly distributed
marginals. The joint cdf of each random vector $(X_1,X_2)$ can be
written as 
\[F(x_1,x_2)=C(F_1(x_1),F_2(x_2))\]
for a copula $C$ with
$F_1$ and $F_2$ being the marginal cdfs. A copula $C$ is called
\emph{symmetric} on a set $A \subset [0,1]$ if $C(u,v) = C(v,u)$ for
all $u,v \in A$. For $A = [0,1]$, symmetry of the copula is equivalent
to the pair $(U_1,U_2)$ being exchangeable.

The following theorem provides a characterization of all dependence
structures that are compatible with the stochastic ordering of the
marginals.

\begin{theorem}
  \label{thm:one}
  Let $(X_1,X_2)$ be a random vector with marginals $X_1$ and $X_2$
  having cdfs $F_1$ and $F_2$, respectively. Then $\prob(X_1\geq
  X_2)=1$ if and only if $F_1(x)\leq F_2(x)$ for all $x\in\R$ and the
  joint cdf of $(X_1,X_2)$ is given by
  \begin{equation}
    \label{eq:1}
    F(x_1,x_2)=
    \begin{cases}
      F_1(x_1), & x_1\leq x_2,\\
      2\tilde{C}(G(x_1),G(x_2))-F_1(x_2),& x_1 > x_2,
    \end{cases}
  \end{equation}
  for all $x_1,x_2\in\R$, where $G$ is given by \eqref{eq:2} and $\tilde{C}$
  is a symmetric copula on the range of $G$ such that
  \begin{equation}
    \label{eq:8}
    \tilde{C}(G(x),G(x))=F_1(x),\quad x\in\R.
  \end{equation}
  % with diagonal section $D$ defined
  % in \eqref{eq:3}.
\end{theorem}
% \begin{remark}
%   The joint distribution of $(X_1,X_2)$ in
%   Theorem~\ref{thm:one} is completely determined by $\tilde{C}$ and
%   $D$. Furthermore, $F(x_1,x_2)=C(F_1(x_1),F_2(x_2))$ for the copula
%   $C$ given by 
%   \begin{equation}
%     \label{eq:C}
%     C(u,v) := \begin{cases} 
%       u, &\text{if  $F_2\circ F_1^{-1}(u) \le v$,} \\ 
%       2  \tilde{C}(G\circ F_1^{-1}(u),G\circ F_2^{-1}(v))-F_1\circ F_2^{-1}(v), 
%       &\text{otherwise}.
%     \end{cases}
%   \end{equation}
% \end{remark}
\begin{proof}[\textnormal{\textbf{Proof}}]%[Proof of Theorem \ref{thm:one}]
  Sufficiency. 
  Let $(V_1,V_2)$ be distributed according to the symmetric bivariate cdf
  $\tilde{C}(G(x_1),G(x_2))$. By construction, $V_1$ and $V_2$ are
  identically distributed with cdf $G$. Furthermore, $\prob(V_1 \le
  x,V_2 \le x) = \tilde{C}(G(x),G(x)) = F_1(x)$. The distribution of
  $(\tilde{X}_1,\tilde{X}_2)=(\max\{V_1,V_2\},\min\{V_1,V_2\})$ is given by \eqref{eq:1}, so
  sufficiency follows from Theorem~\ref{thr:order} because $\prob(\tilde{X}_1 \ge \tilde{X_2}) = \prob(X_1 \ge X_2)$.
  
  \noindent Necessity.
  Let $(V_1,V_2)$ be as in Theorem~\ref{thr:order}. Then, any copula
  $\tilde{C}$ of $(V_1,V_2)$ satisfies
  \begin{align*}
    \tilde{C}(G(x),G(y)) &= \prob(V_1 \le x, V_2 \le y) = \prob(V_1 \le y, V_2 \le x) =  \tilde{C}(G(y),G(x)),\\
    F_1(x) &= \prob(V_1 \le x, V_2 \le x) = \tilde{C}(G(x),G(x))
  \end{align*}
  for all $x,y \in \R$. 
\end{proof}

Diagonal sections of copulas, i.e. the functions that arise as
$\delta(t)=C(t,t)$, $t\in[0,1]$, for some copula $C$ are characterized
by the following properties, see \cite{DuranteMesiarETAL2005}.

\begin{defn}
  A function $\delta:[0,1] \to [0,1]$ is a \emph{diagonal section} if
  \begin{description}
  \item[\textnormal{(D1)}] $\delta(0)=0$, $\delta(1)=1$;
  \item[\textnormal{(D2)}] it is increasing;
  \item[\textnormal{(D3)}] $|\delta(t)-\delta(s)|\le 2|t-s|$ for
    all $t,s \in [0,1]$;
  \item[\textnormal{(D4)}] $\delta(t)\le t$ for all $t \in [0,1]$.
  \end{description}
  % \textbf{(D1)} $\delta(0)=0$, $\delta(1)=1$; \textbf{(D2)} it is
  % increasing; \textbf{(D3)} $|\delta(t)-\delta(s)|\le 2|t-s|$ for
  % all $t,s \in [0,1]$; \textbf{(D4)} $\delta(t)\le t$ for all $t \in
  % [0,1]$.
\end{defn}

For an increasing function $F:\R \to [0,1]$, denote by 
\begin{displaymath}
  F^-(t):=\inf\{x:\; F(x)\geq t\} \in [-\infty,+\infty],\quad t\in[0,1],
\end{displaymath}
the generalized inverse of $F$, where $\inf \emptyset = +\infty$.  Note that
\begin{equation}\label{eq:15}
  F\circ F^-(t)=t
\end{equation}
for all $t$ from the range of $F$, see \cite[Proposition
2.3(4)]{EmbrechtsHofert2013}. For notational convenience, we set
$F(-\infty) = 0$, $F(\infty) = 1$.

The following result follows from the representation \eqref{eq:8} of
the diagonal section of the copula $\tilde{C}$ and identity \eqref{eq:15}.

\begin{corollary}
  \label{cor:one}
  Let $G$ be given by \eqref{eq:2}. The function 
  \begin{equation}
    \label{eq:3}
    D := F_1\circ G^-,
  \end{equation}
  given by the composition of $F_1$ and the generalized inverse of
  $G$, is the restriction of a diagonal section to the range of $G$.
\end{corollary}

We can also provide a converse to Corollary~\ref{cor:one}.

\begin{proposition}
  \label{prop:two}
  Let $\delta$ be a diagonal section. Then there are cdfs $F_1$ and
  $F_2$ such that $F_1 \le F_2$ and $\delta = F_1 \circ G^-$ with
  $G$ defined at \eqref{eq:2}.
\end{proposition}
\begin{proof}[\textnormal{\textbf{Proof}}] 
  We can extend $\delta$ to an increasing function on $\R$ with range
  $[0,1]$. Its generalized inverse $\delta^-$ is left-continuous. For
  $x \in [0,1)$, we define 
  \begin{displaymath}
    \delta^-(x+) = \lim_{y \downarrow x}
    \delta^-(y), \qquad \delta^-(1+) = 1.
  \end{displaymath}
  The function $x\mapsto \delta^-(x+)$ is increasing and
  right-continuous, $\delta(\delta^-(x+))=x$, and $\delta^-(x+) \ge \delta^-(x)$ for all $x \in
  [0,1]$. Let $x \le y$, $x, y \in [0,1]$. Then, by (D3),
\[
0 \le y - x = \delta(\delta^-(y+)) - \delta(\delta^-(x+))  \le 2(\delta^-(y+) - \delta^-(x+)).
\]
  Set $F_1(x) = x$ for $x \in [0,1]$ and $F_2(x) = 2\delta^-(x+) -
  x$ for $x \in [0,1]$. The function $F_2$ is a cdf by the above arguments. By (D4), we obtain $F_1 \le F_2$. We have $G(x) = \delta^-(x+)$, $x \in [0,1]$. It remains to be checked that $G^- = \delta$. The function $\delta$ is constant on $[\delta^-(x),\delta^-(x+)]$ for any $x \in [0,1]$. This implies, for any $t \in (0,1]$,
\begin{align*}
G^-(t) &= \inf\{x:\; G(x) \ge t\} = \inf\{x:\; \delta^-(x+) \ge t\} = \inf\{\delta(\delta^-(x+)):\; \delta^-(x+) \ge t\} \\ &= \delta(\inf\{\delta^-(x+):\; \delta^-(x+) \ge t\}) = \delta(t).\qedhere
\end{align*}
 \end{proof}

Equation~\eqref{eq:8} specifies the diagonal section $\tilde{C}(t,t)$
of the copula $\tilde{C}$ for all $t$ from the range of $G$. If both
$X_1$ and $X_2$ are non-atomic, then this range is $[0,1]$ and so the
diagonal section of $\tilde{C}$ is uniquely specified. 

\begin{example}
  \label{ex:archi}
  An Archimedean generator is a decreasing convex function
  $\psi:[0,\infty) \to [0,1]$ with $\psi(0)=1$ and $\lim_{x \to
    \infty}\psi(x) = 0$, see \cite[Theorem 6.3.2]{SchweizerSklar1983}. Note that we define an Archimedean generator following \cite{McNeilNeslehova2009}.
  For the Archimedean copula $C(u,v) = \psi(\psi^-(v) +
  \psi^-(u))$, $u,v \in [0,1]$, with generator $\psi$, the diagonal
  section is $\delta(t) = \psi(2\psi^-(t))$. 
  
  %The cdf constructed in the proof of Proposition~\ref{prop:two} is $F_2(x) = 2 \psi(\psi^-(x)/2) - x$, $x \in [0,1]$.

  There are many parametric families of Archimedean generators. For
  example, the Gumbel family of copulas is generated by
  $\psi_\theta(t) = \exp(-t^{1/\theta})$, $\theta \in
  [1,\infty)$. Then, the cdf constructed in the proof of Proposition~\ref{prop:two} is $F_2(t) = 2t^a - t$, $t \in [0,1]$, with $a =  2^{-1/\theta}$.
\end{example}

\begin{example}[Identical distributions]
  \label{ex:equal}
  If $X_1$ and $X_2$ are identically distributed, then
  $F_1=F_2=G$. In this case, the diagonal section of $\tilde{C}$ in Theorem~\ref{thm:one} is
  given by $D(t)=t$ for all $t$ from the range of $G$. For $x_1>x_2$, by \eqref{eq:8},
  \begin{displaymath}
    F_1(x_2)=\tilde{C}(G(x_2),G(x_2))\leq 
    \tilde{C}(G(x_1),G(x_2))\leq \tilde{C}(1,G(x_2))
    =G(x_2),
  \end{displaymath}
  hence, $\tilde{C}(G(x_1),G(x_2)) = F_1(x_2)$. Therefore, by \eqref{eq:1}, the joint law of $(X_1,X_2)$ satisfies $F(x_1,x_2)=F_1(\min\{x_1,x_2\})$, meaning that
  $X_1=X_2$ almost surely.
\end{example}

\begin{example}[Discrete distributions]
  \label{ex:discrete}
  Assume that $X_1$ and $X_2$ have discrete distributions, say
  supported on $\{0,1\}$ with masses $p,1-p$ and $q,1-q$,
  respectively, and such that $p\leq q$. The range of $G$ is
  $\{0,(p+q)/2,1\}$. The condition \eqref{eq:8} on $\tilde{C}$ in
  Theorem~ \ref{thm:one} is \[\tilde{C}((p+q)/2,(p+q)/2)= p.\]
  While there are clearly many copulas that satisfy this constraint,
  the condition is sufficient to uniquely determine the joint law of
  $(X_1,X_2)$. By \eqref{eq:1},
  $\prob(X_1=1,X_2=0)=2\tilde{C}(1,(p+q)/2)-2p=q-p$.
\end{example}

\begin{example}[Disjoint supports]
  \label{ex:disjoint}
  Assume that $X_1$ is uniformly distributed on $[1,2]$ and $X_2$ on
  $[0,1]$. In this case, all kinds of dependency structures between
  $X_1$ and $X_2$ are allowed. For $x_1\in[1,2]$ and $x_2\in[0,1]$,
  \eqref{eq:1} yields that 
  \begin{displaymath}
    F(x_1,x_2)=2\tilde{C}(x_1/2,x_2/2). 
  \end{displaymath}
  As prescribed by \eqref{eq:8}, $\tilde{C}(t,t)=0$ for $t\in[0,1/2]$
  and $\tilde{C}(t,t)=2t-1$ for $t\in(1/2,1]$. This is the diagonal section of the Fr\'echet-Hoeffding lower bound.
  
  It is not a contradiction that any copula $C$ yields a possible bivariate law $F(x_1,x_2) = C(F_1(x_1),F_2(x_2))$ of $(X_1,X_2)$ such that $X_1 \ge X_2$ almost surely, but in the representation of Theorem~\ref{thm:one}, there are restrictions on the diagonal of the symmetric copula $\tilde{C}$. The copula $\tilde{C}$ is the copula of the random permutation $(V_1,V_2)$ of $(X_1,X_2)$, and as such it cannot put any mass on the squares $[0,1/2]^2$ or $[1/2,1]^2$.
\end{example}

% \begin{example}[Subset intervals]
%   \label{ex:subset}
%   Assume that $X_1$ is uniformly distributed on $[0,2]$ and $X_2$ is
%   uniformly distributed on $[0,1]$. 
% \end{example}

\section{Pointwise bounds on the joint cdf}
\label{sec:maxim-indep-pairs}

% Theorem~\ref{thm:one} shows that the family of all
% bivariate distributions of random vectors $(X_1,X_2)$ with marginals
% $F_1$ and $F_2$, and such that $\prob(X_1 \ge X_2) = 1$ is in
% one-to-one correspondence with all symmetric copulas with diagonal
% section $D$ as defined in Proposition~\ref{prop:one}. 

By Theorem~\ref{thm:one}, the range of all possible bivariate cdfs of
random vectors $(X_1,X_2)$ with given marginals $F_1$ and $F_2$ and
such that $X_1\geq X_2$ a.s.\ depends on the choice of a symmetric
copula $\tilde{C}$ satisfying \eqref{eq:8}, equivalently, having the
diagonal section \eqref{eq:3} on the range of $G$. For a general
diagonal section $\delta$, the following result holds.

\begin{theorem}[\cite{nel:ques:rod:04,KlementKolesarova2005}]
  \label{thm:KK05}
  Each copula $\tilde{C}$ with diagonal section $\delta$
  satisfies
  \begin{displaymath}
    B_\delta(u,v)\leq \tilde{C}(u,v),\quad u,v\in[0,1],
  \end{displaymath}
  where 
  \begin{equation}
    \label{eq:lowerB}
    B_\delta(u,v) := \min\{u,v\}
    - \inf\{t - \delta(t) : t \in [\min\{u,v\},\max\{u,v\}]\},
    %- \inf_{\min\{u,v\}\leq t\leq \max\{u,v\}} (t - D(t)), 
    \quad u,v\in[0,1],
  \end{equation} 
  is the \emph{Bertino copula}.
%  , and 
%  \begin{equation}
%    \label{eq:upper}
%    K_\delta(u,v) := \min\left\{u,v,\frac{\delta(u)+\delta(v)}{2}\right\},
%    \quad u,v\in [0,1].
 % \end{equation} 
 % is the \emph{diagonal copula}. 
  The copula $B_\delta$  has diagonal section $\delta$.
\end{theorem}

% Let $\mathcal{C}_D$ denote the family of all possible copulas of
% $(X_1,X_2)$, see \eqref{eq:C}.  Combining Theorem~\ref{thm:KK05} with
% Theorem~\ref{thm:one} yields the following result.

% \begin{corollary} 
%   \label{cor:KK05}
%   Let $D$ be a diagonal section as defined in
%   Proposition~\ref{prop:one}.  The family $\mathcal{C}_D$ is bounded
%   from above by the \emph{Fr\'echet-Hoeffding upper bound}
%   \begin{displaymath}
%     M(u,v) := \min(u,v), \quad u,v,\in[0,1],
%   \end{displaymath}
%   and from below by the copula $W_D(u,v)$, such that 
%   \begin{multline}
%     \label{eq:Wdelta}
%     W_D(F_1(x_1),F_2(x_2)) \\:=\begin{cases} 
%       F_1(x_1), &\text{if  $x_1\le x_2$,} \\ 
%       F_2(x_2) - \inf\limits_{F_1(x_2)\leq s\leq F_2(x_2)}
%       (F_2(F_1^{-1}(s)) - s), &\text{otherwise.}
%     \end{cases}
%   \end{multline}
%   % \begin{multline}
%   %   \label{eq:Wdelta}
%   %   W_D(u,v) :=\\ \begin{cases} 
%   %     u, &\text{if  $F_1^{-1}(u)\le F_2^{-1}(v)$,} \\ 
%   %     v - \inf\limits_{F_2^{-1}(v)\leq s\leq F_1^{-1}(v)}
%   %     (F_2(s) - F_1(s)), &\text{otherwise.}
%   %   \end{cases}
%   % \end{multline}
% \end{corollary}
% % \begin{proof}[\textnormal{\textbf{Proof}}]
% % It suffices to observe that if $\tilde{C} \le \tilde{C}'$ pointwise for some symmetric copulas $\tilde{C}$, $\tilde{C}'$ with diagonal section $D$, the resulting copulas $C$, $C'$ in \eqref{eq:C} also satisfy $C \le C'$ pointwise. Then, the claim follows by plugging in the formulas in Theorem \ref{thm:KK05}.
% % \end{proof}

Denote
\begin{equation}
  \label{eq:11}
  H(x):=F_2(x)-F_1(x), \quad x\in\R.
\end{equation}

\begin{theorem} 
  \label{cor:KK05}
  Each random vector $(X_1,X_2)$ with marginal cdfs $F_1$ and $F_2$,
  and such that $X_1\geq X_2$ a.s., has a joint cdf $F$ satisfying
  \begin{equation}
    \label{eq:4}
    L(x_1,x_2)\leq F(x_1,x_2)\leq \min\{F_1(x_1),F_2(x_2)\},\quad
    x_1,x_2\in\R,
  \end{equation}
  where both bounds are attained, and 
  \begin{equation}
    \label{eq:5}
    L(x_1,x_2) :=
    \begin{cases}
      F_1(x_1), & x_1\leq x_2,\\
      F_2(x_2)-\inf_{x_2\leq s\leq x_1} H(s), & x_1 > x_2,
    \end{cases}
  \end{equation}
\end{theorem}
\begin{proof}[\textnormal{\textbf{Proof}}]
  The upper bound in \eqref{eq:4} is the Fr\'echet--Hoeffding
  one; it corresponds to complete dependence between $X_1$ and $X_2$,
  so that $X_1=F_1^-(U)$ and $X_2=F_2^-(U)$ for a standard
  uniformly distributed random variable $U$.

  For the lower bound, let $\delta$ be a diagonal section which is
  equal to $D$ at \eqref{eq:3} on the range $R_G$ of $G$. We
  continuously extend $D$ to the closure $\operatorname{cl}(R_G)$ of
  $R_G$. Continuity of $\delta$ implies that $\delta$ is equal
  $D$ on $\operatorname{cl}(R_G)$. The function $\delta$ is bounded below by
  $\delta_G$ defined as 
  \begin{displaymath}
    \delta_G(x)=\max\{D(x^-),D(x^+)-2(x^+-x)\},
  \end{displaymath}
  where $x^-=\sup R_G\cap[0,x]$ and $x^+=\inf R_G\cap[x,1]$. This
  function is itself a diagonal section which is equal to $D$ on
  $\operatorname{cl}(R_G)$. For all $u,v \in [0,1]$ it holds that
  $B_{\delta}(u,v) \ge B_{\delta_G}(u,v)$. Therefore,
  Theorem~\ref{thm:KK05} and \eqref{eq:1} imply for $x_1>x_2$,
  \begin{equation}
    \label{eq:43}
    F(x_1,x_2)\geq 2B_{\delta_G}(G(x_1),G(x_2))-F_1(x_2)
    =2(G(x_2)-\inf\{t-\delta_G(t): t\in[G(x_2),G(x_1)]\})-F_1(x_2). 
  \end{equation}
  Since 
  \begin{displaymath}
    t - \delta_G(t) = \min\{t^- - D(t^-) + t-t^-,t^+ - D(t^+)+t^+ - t\},
  \end{displaymath}
  we can restrict the infimum in \eqref{eq:43} to $R_G$. Hence, 
  \begin{align*}
    F(x_1,x_2)&\geq
    F_2(x_2)-2\inf\{G(x)-F_1\circ G^-\circ G(x):
    x\in[x_2,x_1]\} =L(x_1,x_2). 
  \end{align*}
  The last equality holds because $G^-\circ G(x) \le x$ always holds
  (see \cite{EmbrechtsHofert2013}) and $G^-\circ G(x) < x$ only
  happens if there is an $\varepsilon > 0$ such that $G$ is constant
  on $(x-\varepsilon,x]$. But if $G$ is constant on some interval,
  then $F_1$ is necessarily also constant on this interval.
\end{proof}

The lower bound in \eqref{eq:4} corresponds to the Bertino copula and
so yields the least possible dependence between $X_1$ and $X_2$. If
the function $H$ at \eqref{eq:11} is unimodal, this corresponds to the
assumption that the Bertino copula \eqref{eq:lowerB} is simple,
compare \cite{FredricksNelsen2002}. The unimodality condition (which
also appears in Theorem~\ref{thm:SpearmanW}) applies in many examples,
and simplifies the structure of the distribution $L$ considerably.

\begin{corollary}
  \label{cor:unimodal}
  Assume that the function $H$ is unimodal, that is, $H$ increases on
  $(-\infty,r]$ and decreases on $[r,\infty)$ for some $r$.  Then 
  \begin{equation}
    \label{eq:10}
    L(x_1,x_2) =
    \begin{cases}
      F_1(x_1), & x_1\leq x_2,\\
      F_1(x_1)-\min\{F_1(x_1)-F_1(x_2),F_2(x_1)-F_2(x_2)\}, & x_1 > x_2.
    \end{cases}
  \end{equation}
 \end{corollary}
\begin{proof}[\textnormal{\textbf{Proof}}] Let $x_1 > x_2$.
  By the unimodality, the infimum of $H$ over $[x_2,x_1]$ is attained
  at one of the end-points $x_2$ or $x_1$. Therefore, 
  \begin{align*}
F_2(x_2)-\inf_{x_2\leq s\leq x_1} H(s) &= F_2(x_2) - \min\{F_2(x_1)-F_1(x_1),F_2(x_2) - F_1(x_2)\} \\&=  F_1(x_1) - \min\{F_2(x_1)-F_1(x_1)-F_2(x_2) + F_1(x_1),F_2(x_2) - F_1(x_2)-F_2(x_2) + F_1(x_1)\},
  \end{align*}
  which yields \eqref{eq:10}.
  \end{proof}

\begin{example}[Disjoint supports -- Example \ref{ex:disjoint} continued]\label{ex:disjoint2}
  We assume that $X_1$ is uniformly distributed on $[1,2]$ and $X_2$ on
  $[0,1]$. Then, 
  \begin{equation}\label{eq:disjoint21}
  H(x) = F_2(x) - F_1(x) = \begin{cases} x, & x \in [0,1],\\
  2 - x, &x \in [1,2], \\
  0, & \text{otherwise},\end{cases}
  \end{equation}
  which is clearly unimodal. Therefore, by Corollary \ref{cor:unimodal},
  and for $x_1, x_2 \in [0,2]$
  \begin{equation}\label{eq:disjoint22}
    L(x_1,x_2) =
    \begin{cases}
    0, & x_1 \in [0,1],\\
     x_1 - 1, & x_1,x_2 \in [1,2],\\
     \max\{x_1+x_2-2,0\}, & x_1 \in [1,2], x_2 \in [0,1].\\
    \end{cases}
    = \max\{F_1(x_1) + F_2(x_2) - 1,0\},
  \end{equation}
  which corresponds to choosing the Fr\'echet-Hoeffding lower bound as
  the dependence structure for $(X_1,X_2)$.
\end{example}

As shown by Rogers in \cite{rog99}, if $H$ is unimodal, the distribution given by
\eqref{eq:10} maximizes the payoff (or transportation cost)
$\E\phi(|X_1-X_2|)$ over all strictly convex decreasing functions
$\phi:\R_+\mapsto\R_+$. Without unimodality assumption, the joint
distribution maximizing the payoff is given by
\begin{displaymath}
  P(x_1,x_2):=
  \begin{cases}
    F_1(x_1), & x_1\leq x_2,\\
    \sup_{v\leq x_2}\Big[F_2(v)-\inf_{v\leq s\leq x_1} H(s)\Big], & x_1 > x_2.
  \end{cases}
\end{displaymath}
This joint distribution satisfies $L(x_1,x_2)\leq P(x_1,x_2)\leq
\min\{F_1(x_1),F_2(x_2)\}$; it provides the joint distribution with
the largest mass concentrated on the diagonal, see
\cite[Th.~7.2.6]{rac:rus98}.

The following result concerns the support of the random vector with
distribution $L$ in the case when the cdfs $F_1$ and $F_2$ are
continuous. In the general case, the support of $L$ is more intricate
to describe.

\begin{lemma}
  \label{lemma:support} 
  Suppose that $F_1$ and $F_2$ are continuous.
  The support of the distribution $L$ given at
  \eqref{eq:5} is the set
  \begin{align}
    \label{eq:9}
    A&=\Big\{(x_1,x_2)\in\R^2:\; 
    x_2 < x_1,\; H(x_2)=H(x_1)< H(s) \;\text{for all $s \in (x_2,x_1)$} \Big\}\notag \\
    & \quad \cup \Big\{(x_1,x_2)\in\R^2:\; x_2 < x_1,\; H(x_2)=H(x_1)\le
    H(s) \;\text{for all $s \in (x_2,x_1)$}\notag \\& \qquad \quad\text{and 
      % there exists $\eps > 0$ such that
      $H(s) < H(x_1)=H(x_2)$ for $s \in
      (x_2-\eps,x_2) \cup (x_1,x_1+\eps)$ for some $\eps>0$}
    \Big\}\notag 
    \\& \quad \cup \Big\{(x,x) \in \R^2:\; x \in (S_1 \cap S_2)\backslash T\Big\},
  \end{align}
  where $S_1$ and $S_2$ are the supports of the distributions $F_1$ and
  $F_2$, respectively, and
  \begin{displaymath}
    T= \Big\{ x \in \partial S_1 \cap \partial S_2:\; 
    \text{$(x,x+\eps)\cap S_1 = \emptyset$ and $(x-\eps,x) \cap S_2 =
      \emptyset$ for some $\eps>0$}\Big\}.
  \end{displaymath}
  Here, $\partial S_i$, denotes the topological boundary of $S_i$,
  $i=1,2$.
\end{lemma}
\begin{proof}[\textnormal{\textbf{Proof}}] Let $(X_1,X_2)$ have distribution $L$. Let $x \in \R$ and $\eps > 0$. Then, 
  \begin{align}\label{eq:765}
    \prob((X_1,X_2)\in (x-\eps,x+\eps]^2)&=F_1(x+\eps)-F_2(x-\eps)+\inf_{x-\eps\leq s\leq x+\eps}
    H(s)\\
   & =\inf_{x-\eps\leq s\leq x+\eps} \big(F_1(x+\eps) - F_1(s) + F_2(s) - F_2(x-\eps)\big)\nonumber
  \end{align}
The right hand side can only be strictly positive if 
  \begin{displaymath}
    \min\{F_1(x+\eps)-F_1(x-\eps),F_2(x+\eps)-F_2(x-\eps)\}>0,
  \end{displaymath}
  which is the case whenever $\prob(X_1\in(x-\eps,x+\eps])>0$ and
  $\prob(X_2\in(x-\eps,x+\eps])>0$. Thus, only points $(x,x)$ with $x$ belonging
  to $S_1$ and $S_2$ may be in the diagonal parts of the support of $L$. If $x \in T$ and $\eps > 0$ is small enough, then
  \[
  \inf_{x-\eps\leq s\leq x+\eps} \big(F_1(x+\eps) - F_1(s) + F_2(s) - F_2(x-\eps)\big) = \big(F_1(x+\eps) - F_1(x) + F_2(x) - F_2(x-\eps) = 0,
  \]
  so $(x,x)$ cannot belong to the support of $L$.
  
  Conversely, assume that $x \in S_1 \cap S_2$. Since $H$ is
  continuous, the infimum in \eqref{eq:765} is attained at some
  $s_0(\eps) \in [x-\eps,x+\eps]$, if $s_0(\eps) < x$, then $x \in
  (s_0(\eps),x+\eps)$, hence $F_1(x+\eps) - F_1(s_0(\eps)) > 0$
  because $x \in S_1$. One can argue analogously if $s_0(\eps) >
  x$. If $s_0(\eps) = x$, then we distinguish two cases. If $x \in
  \operatorname{int} S_1$ or $x \in \operatorname{int} S_2$, then one
  can argue as previously. Here $\operatorname{int}(S_i)$ denotes the
  interior of $S_i$, $i\in\{1,2\}$. If $x \in \partial S_1 \cap \partial
  S_2$ and
  \[
  \inf_{x-\eps\leq s\leq x+\eps} \big(F_1(x+\eps) - F_1(s) + F_2(s) - F_2(x-\eps)\big) = \big(F_1(x+\eps) - F_1(x) + F_2(x) - F_2(x-\eps) = 0,
  \]
  then $x \in T$ which yields the claim concerning the diagonal part of the support of $L$.

  Now assume that $x_2<x_1$ and $0 < \eps < (x_1-x_2)/2$. Then 
  \begin{equation}
    \label{eq:6}
    \prob((X_1,X_2)\in (x_1-\eps,x_1+\eps]\times(x_2-\eps,x_2+\eps])
    =\min\{a,b,c\}+b-\min\{a,b\}-\min\{b,c\},
  \end{equation}
  where
  \[
    a:=\inf_{x_2-\eps\leq s\leq x_2+\eps} H(s),\quad b:=\inf_{x_2+\eps\leq s\leq x_1-\eps} H(s),\quad c:=\inf_{x_1-\eps\leq s\leq x_1+\eps} H(s).
  \]
  The probability in \eqref{eq:6} is strictly positive if and only if
  $a\leq c< b$ or $c\leq a< b$. The point $(x_1,x_2)$ belongs to the support of $L$ if and only if $\max\{a,c\} < b$ for all $\eps>0$ small enough. Letting $\eps$ converge to zero, we find that a necessary condition is that 
  \[
  \max\{H(x_1),H(x_2)\} \le \min\{H(x_2),H(x_1)\},
  \]
  hence $H(x_1) = H(x_2)$. It is not hard to check that the conditions on $x_2$ and $x_1$ in $A$ are necessary and sufficient to ensure that $\max\{a,c\} < b$ is fulfilled for all $\eps > 0$ small enough. 
\end{proof}

The set $A$ from \eqref{eq:9} is illustrated in the top-left panel of Fig.~\ref{fig:ex}
using points sampled from $L$. 

\begin{example}[Disjoint supports -- Example \ref{ex:disjoint} continued]
  We assume that $X_1$ is uniformly distributed on $[1,2]$ and $X_2$
  on $[0,1]$. The function $H$ and the distribution $L$ are given at
  \eqref{eq:disjoint21} and \eqref{eq:disjoint22}, respectively. The
  support of $L$ is given by
  \[
  \Big\{(2-x,x) \in \R^2:\; x \in [0,1]\Big\}.
  \]
  This follows from Example~\ref{ex:disjoint2} or Lemma~\ref{lemma:support}. The set $A$ at
  \eqref{eq:9} consists of three parts. The first set is
  $\{(2-x_2,x_2)\in \R^2:\; x_2 \in [0,1)\}$, the second set is empty,
  and the third set is $\{(x,x) \in \R^2:\; x = 1\}$ because $T$ is the empty set.
\end{example}

\section{Nonparametric correlation coefficients}
\label{sec:nonp-corr-coeff}

Dependence measures quantitatively summarize the degree of dependence
between two random variables $X_1$ and $X_2$.  Kendall's tau and
Spearman's rho are arguably the two most well-known measures of
association whose sample versions are purely based on ranks. If the
marginal distributions of $X_1$ and $X_2$ are continuous then the
population versions of Kendall's tau and Spearman's rho only depend on
the copula of $(X_1,X_2)$. In this section, we assume that both, $F_1$
and $F_2$ are continuous, and hence the copula of $(X_1,X_2)$ is
uniquely defined. We refer the reader to \cite{Neslehova2007} for
details concerning problems that arise in the case of arbitrary
marginal distributions.

Kendall's tau of a copula $C$ is given by 
\begin{equation}
  \label{eq:Ktau}
  \tau := 4\int_{[0,1]^2} C(u,v) \diff C(u,v) - 1,
\end{equation}
and, Spearman's rho is given by
\begin{equation}
  \label{eq:Srho}
  \rho := 12 \int_{[0,1]^2} C(u,v)\diff u \diff v - 3
  = 12 \int_{[0,1]^2} uv \diff C(u,v) - 3.
\end{equation}
These two correlation coefficients are monotonic with respect to
pointwise, or, concordance ordering of copulas
\cite{Neslehova2007}. They take the value one for the joint cdf given
by the upper bound in \eqref{eq:4}. For the copula
$C(u,v):=L(F_1^-(u),F_2^-(v))$ with $L$ given by \eqref{eq:5}, these
dependence measures attain their lowest values calculated as follows.

\begin{theorem}
  \label{thm:KendallW}
  Suppose that $F_1$ and $F_2$ are continuous. The smallest possible
  Kendall's tau of $(X_1,X_2)$ that satisfies the conditions of
  Theorem~\ref{thm:one} is
  \begin{displaymath}
    \tau=4\;\E F_1(X_2)-1.
  \end{displaymath}
\end{theorem}
\begin{proof}[\textnormal{\textbf{Proof}}]
  Writing $C(u,v)=L(F_1^-(u),F_2^-(v))$ yields that
  \begin{displaymath}
    \tau=4\;\E L(X_1,X_2)-1,
  \end{displaymath}
  where $(X_1,X_2)$ has cdf $L$ given by \eqref{eq:5}. By
  Lemma~\ref{lemma:support}, the support of $(X_1,X_2)$ is given by
  the set $A$ at \eqref{eq:9}.  On the set $A$,
  \begin{displaymath}
    L(X_1,X_2)=F_2(X_2)-(F_2(X_2)-F_1(X_2)),
  \end{displaymath}
  hence the result.
\end{proof}

\begin{theorem} 
  \label{thm:SpearmanW}
  Suppose that $X_1$ and $X_2$ have continuous cdfs $F_1$ and $F_2$
  with the same support $[x^L,x^U]$ and that the function $H(s)=F_2(s)
  - F_1(s)$ from \eqref{eq:11} is unimodal, strictly increases on
  $(x^L,r]$ and strictly decreases on $[r,x^U)$ for some $r$.
  Then the smallest possible Spearman's rho of $(X_1,X_2)$ satisfying
  the conditions of Theorem~\ref{thm:one} is
  \begin{align}
    \label{eq:sp-rho}
    \rho= 12\; \Bigg[\int_{x^L}^r F_1(s)F_2(s)\diff F_1(s)
    &+ \int_r^{x^U} F_1(s)F_2(t(s))\diff F_1(s) + \int_r^{x^U} F_1(s)(F_2(s)-F_2(t(s)))\diff F_2(s)\Bigg]
    -3,
  \end{align}
  where $F_2(t(s))-F_1(t(s))=F_2(s)-F_1(s)$ and $t(s)<s$ for $s\in(r,x^U]$ and
  $t(r)=r$. 
\end{theorem}
\begin{proof}[\textnormal{\textbf{Proof}}]
  Spearman's rho is given by 
  \begin{displaymath}
    \rho=12\;\E \big(F_1(X_1)F_2(X_2)\big)-3.
  \end{displaymath}
  The set $A$ from \eqref{eq:9}, consists of 3 pieces: $\{(s,s):\;
  s\in(x^L,r]\}$ with the push-forward of $F_1(s)$,
  $s\in(x^L,r]$; $\{(s,s):\; s\in(r,x^U)\}$ with the
  distribution being the image of the measure on $(r,x^U)$ with
  push-forward of $F_2(s)-F_2(r)$; and $\{(s,t(s)):\;
  s\in(r,x^U)\}$ with the distribution being the push-forward of
  $H(r)-H(s)$. The push-forward is the image of the measure on the
  line by the specified map, e.g., the third part if the image of the
  measure $\mu$ on $(r,x^U)$ with $\mu((r,s])=H(r)-H(s)$ under
  the map $s\mapsto(s,t(s))$. 

  The result is obtained by splitting the above expectation into these
  3 parts.
\end{proof}  

% and the support of the distribution with cdf $L$ is contained in the union of the
%   diagonal in $\R^2$ and the set 
%   \[
%   \{(t(x),x):\; x\in[r,\infty)\cap
%   S\},
%   \]
%    where
%   \begin{displaymath}
%     t(x):=\sup\{y:\; H(y)=H(x)\},
%   \end{displaymath}
%   and $S$ is the set of $x\in(r,\infty)$ such that $H(s)>H(x)$ whenever 
%   $t(x)<s<x$.

% \end{proof}

\section{Maximum entropy distributions}
\label{sec:maxim-entr-solut}

We assume that both $X_1$ and $X_2$ have full supports on a (possibly
infinite) interval $[x^L,x^U]$ and that their cdfs $F_1$ and $F_2$ are
absolutely continuous with densities $f_1$ and $f_2$.  Amongst all
joint absolutely continuous laws $F$ of $(X_1,X_2)$ with given
marginals $F_1$ and $F_2$ and such that $X_1 \ge X_2$ a.s., we 
characterize those maximizing the differential entropy (see
\cite[Ch.~8]{cov:thom06}) given by
\begin{equation}
  \label{eq:Shannon}
  \ent(F) :=  -\int f(x_1,x_2)\log f(x_1,x_2)\diff x_1\diff x_2.
\end{equation}
These copulas correspond to the least informative (most random) joint
distributions, equivalently, to the distributions minimizing the
Kullback--Leibler divergence with respect to the uniform
distribution. We use the common convention $0\log 0 = 0$. Independently of our work, maximum entropy distributions of order statistics in the multivariate case have been studied in \cite{ButuceaDelmasETAL2018}.

Note that the function $G$ from \eqref{eq:2} is absolutely continuous
with density $g = (f_1 + f_2)/2$.  By Theorem~\ref{thm:one}, the joint
law $F$ of $(X_1,X_2)$ is absolutely continuous if and only if the
associated symmetric copula $\tilde{C}$ is absolutely continuous. We
denote its density by $\tilde{c}$.  By the symmetry of $\tilde{C}$,
\begin{equation}
  \label{eq:HFHC}
  \ent(F) = -\int_{[0,1]^2} \tilde{c}(z_1,z_2)
  \log \tilde{c}(z_1,z_2)\diff z_1\diff z_2
  - \log 2 - 2 \int_{x^L}^{x^U} g(z) \log g(z)\diff z. 
\end{equation}
Therefore, maximizing $\ent(F)$ over all $F$ is equivalent to
maximizing $\ent(\tilde{C})$ over all symmetric copulas $\tilde{C}$
with diagonal section $D=F_1 \circ G^-$. Note that the smallest
entropy $-\infty$ arises as the limit by considering absolutely
continuous distributions approximating the distribution of
$X_1=F_1^-(U)$ and $X_2=F_2^-(U)$ for a uniformly distributed $U$.

Butucea et al.~\cite{ButuceaDelmasETAL2015} characterize the maximum entropy copula
with a given diagonal section $\delta$. We recall some of their
notation in order to be able to state our result. For a diagonal
section $\delta$ with $\delta(t) < t$ for all $t \in (0,1)$, define
for $u, v \in [0,1]$, $u \le v$,
\begin{equation}
  \label{eq:cdelta_simple}
  \bar{c}_\delta(u,v) := \frac{\delta'(v)(2-\delta'(u))}
  {4\sqrt{(v-\delta(v))(u-\delta(u))}}
  \exp\left(-\frac{1}{2}\int_{u}^v \frac{1}{s-\delta(s)}\diff s\right), 
\end{equation}
and for $u \ge v$, set $\bar{c}_\delta(u,v) =
\bar{c}_\delta(v,u)$. Butucea et al.~\cite[Proposition~2.2]{ButuceaDelmasETAL2015}
show that $\bar{c}_\delta$ is the density of a symmetric copula with
diagonal section $\delta$. Note that the derivative $\delta'$ of
$\delta$ exists almost everywhere as $\delta$ is Lipschitz
continuous. For a general diagonal section $\delta$, due to its
continuity, the set $\{t \in [0,1]: \delta(t) < t\}$ is the union of
disjoint open intervals $(\alpha_j,\beta_j)$, $j \in J$ for an at
most countable index set $J$. Note that $\delta(\alpha_j)=\alpha_j$
and $\delta(\beta_j)=\beta_j$. For $u,v \in [0,1]$, define
\begin{equation}
  \label{eq:cdelta_general}
  c_\delta(u,v) := \sum_{j \in J}\frac{1}{\beta_j - \alpha_j}
  \bar{c}_{\delta_j}\left(\frac{u-\alpha_j}{\beta_j - \alpha_j},
    \frac{v-\alpha_j}{\beta_j - \alpha_j}\right)\one_{(\alpha_j,\beta_j)^2}(u,v),
\end{equation}
where $\bar{c}$ is given at \eqref{eq:cdelta_simple}, and 
\begin{displaymath}
  \delta_j(t) := \frac{\delta(\alpha_j
    + t(\beta_j-\alpha_j))-\alpha_j}{\beta_j - \alpha_j},
  \quad t \in [0,1].
\end{displaymath}
Based on the results of \cite{ButuceaDelmasETAL2015}, we arrive at
the following theorem. Recall that $H=F_2-F_1$.

\begin{theorem}
  \label{thm:maxentro}
  Let $(X_1,X_2)$ be a random vector with marginals $X_1$ and $X_2$
  satisfying $\mathbb{P}(X_1 \ge X_2) = 1$. Suppose that $X_1$ and
  $X_2$ have identical support being a (possibly unbounded) interval
  $[x^L,x^U]$, and that their cdfs $F_1$ and $F_2$ are absolutely
  continuous with densities $f_1$ and $f_2$.  If
  \begin{equation}
    \label{eq:7}
    -\int_{x^L}^{x^U}\log H(z) \diff G(z) < \infty,
  \end{equation}
  then
  \begin{displaymath}
    -\infty < \sup_F \ent(F) = \max_F \ent(F) < \infty,
  \end{displaymath}
  where the supremum is taken over all possible joint laws of
  $(X_1,X_2)$. The maximum is attained when the joint density of
  $(X_1,X_2)$ is given by
  \begin{displaymath}
    f(x_1,x_2) = 2 c_D(G(x_1),G(x_2))g(x_1)g(x_2),
    \quad x^L\leq x_2\leq x_1\leq x^U,
  \end{displaymath}
  and $c_D$ is defined at \eqref{eq:cdelta_general}. 
  If \eqref{eq:7} does not hold, then $\sup_F \ent(F)=-\infty$.
\end{theorem}
\begin{proof}[\textnormal{\textbf{Proof}}]
  By substitution, 
  \begin{displaymath}
    -\int_0^1 \log(t-D(t))\diff t
    = -\int_{x^L}^{x^U}\log \frac{H(z)}{2}\diff G(z).
  \end{displaymath}
  The result now follows from \cite[Th.~2.5]{ButuceaDelmasETAL2015}
  in combination with \eqref{eq:HFHC} and Theorem~\ref{thm:one}.
\end{proof}

The condition $D(t) < t$ for all $t \in (0,1)$ is equivalent to
$F_2(z) > F_1(z)$ for all $z \in (x^L,x^U)$. If this condition holds,
then $c_D = \bar c_D$ and the formula for the entropy maximizing $f$
in Theorem~\ref{thm:maxentro} simplifies to
\begin{displaymath}
  f(x_1,x_2) = \frac{f_1(x_1)f_2(x_2)}{\sqrt{H(x_1){H(x_2)}}}
  \exp\left(-\int_{x_2}^{x_1}\frac{1}{H(s)}\diff G(s)\right)
\end{displaymath}
for $x_1,x_2 \in [x^L,x^U]$, $x_1 \ge x_2$.

\section{Examples}
\label{sec:examples}

\begin{example}\label{ex:6n1}
  Let $X_2$ be uniformly distributed on $[0,1]$, and let $X_1$ be distributed as
  the maximum of $X_2$ and another independent uniformly distributed
  random variable, that is, $F_1(x)=x^2$. In this case $H(s)=s-s^2$ is
  unimodal with the maximum at $r=1/2$, and $t(s)=1-s$ for
  $s\in(r,1]$. The top-left panel of Fig.~\ref{fig:ex} shows a sample from
  the distribution $L$. It is easily seen that these values belong to the
  set $A$ given by \eqref{eq:9} which consists here of the diagonal of the square $[0,1]^2$ and the lower part of the off-diagonal.

  \begin{figure}[htbp]
    \centering
    \includegraphics[width=0.45\textwidth]{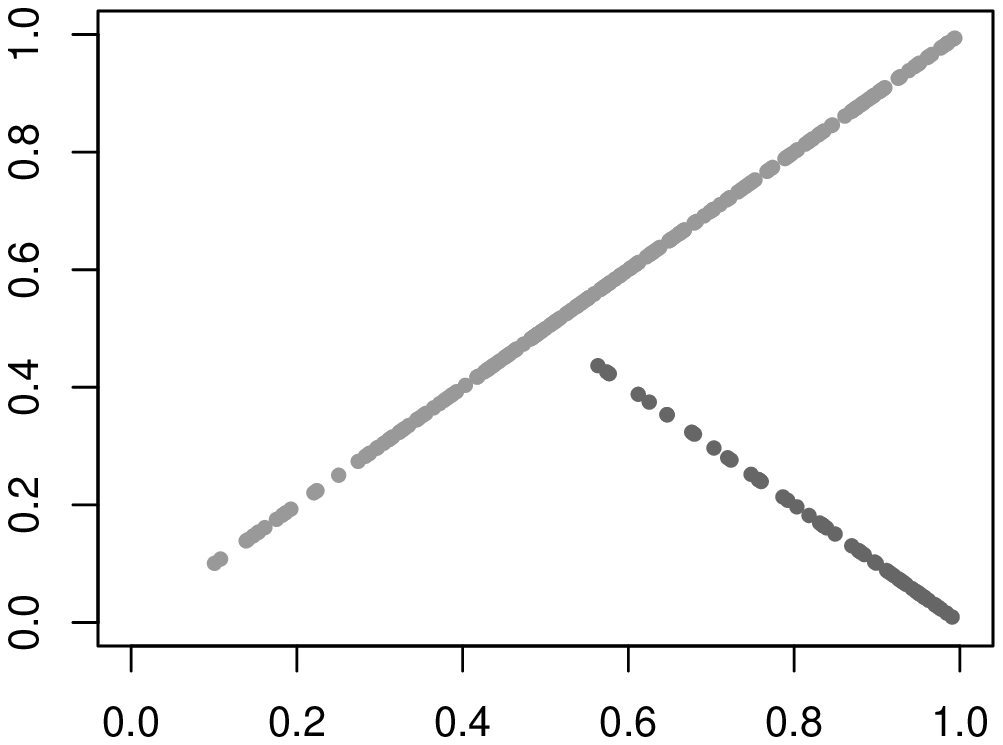} \includegraphics[width=0.45\textwidth]{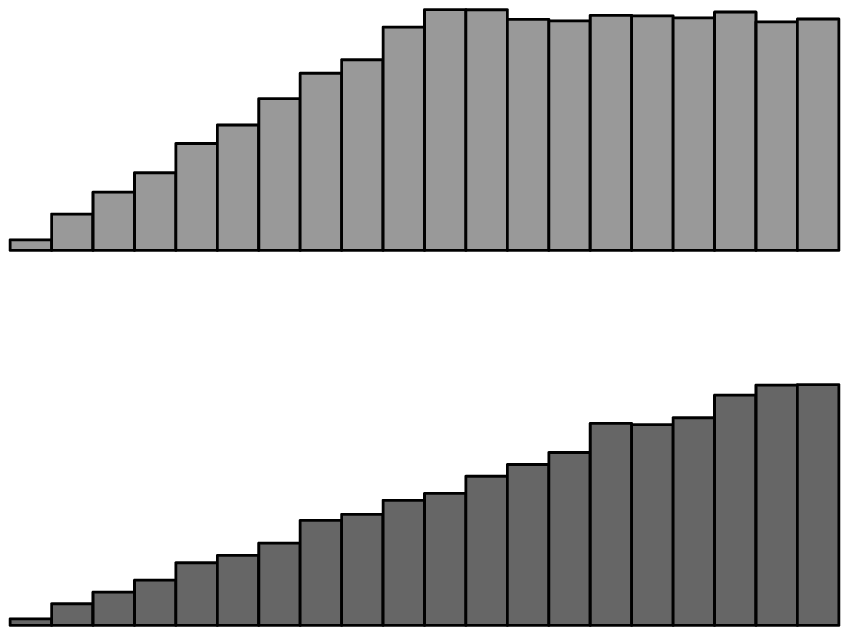}
     \includegraphics[width=0.45\textwidth]{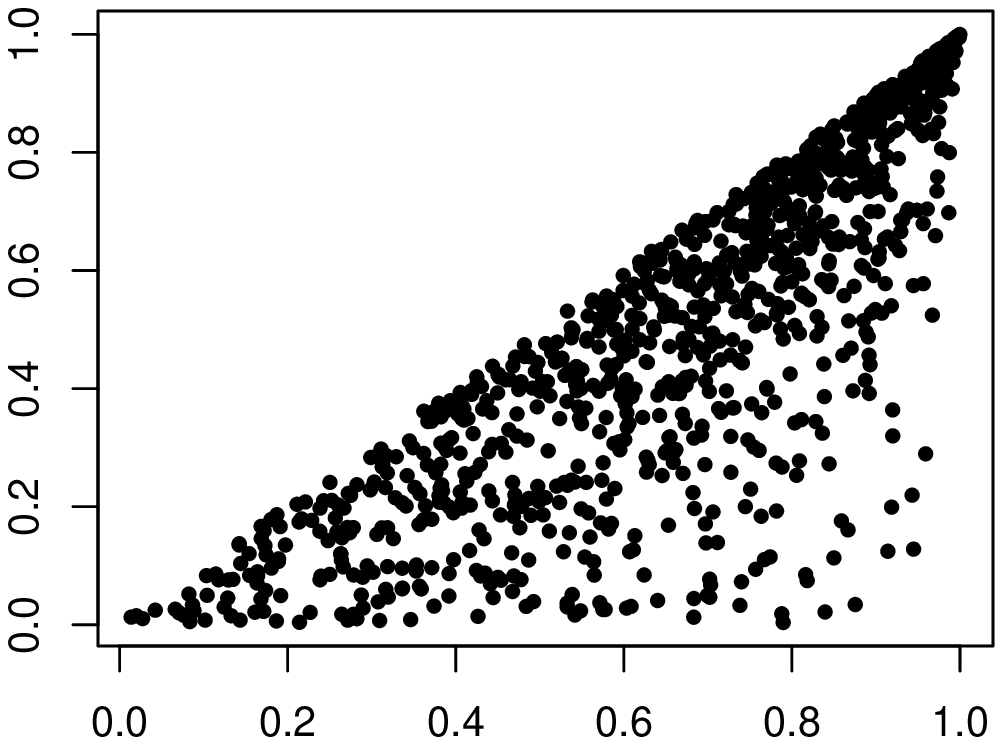}
    \caption{Top-left panel: Sample of size $n=300$ of points (dark and light gray) from the maximally independent joint
      distribution $L$ for Example~\ref{ex:6n1}; Top-right panel: Histograms of two subsamples of a sample of $L$ of size $n=50000$ depending on whether the points are on the diagonal (light gray) or the off-diagonal (dark gray); Bottom panel: Sample of size $n=1000$ of points from
    the entropy maximizing distribution in
    Example~\ref{ex:6n1}. \label{fig:ex}}
  \end{figure}
  
  The smallest values for Kendall's tau and Spearman's
  rho are $1/3$ and $1/4$, respectively. The joint density with the
  maximal entropy is given by 
  \begin{equation}\label{eq:maxe61}
    f(x_1,x_2)=\frac{2(1-x_1)}{(1-x_2)^2},\quad 0\leq x_2\leq x_1\leq 1.
  \end{equation}
  A sample from this distribution is shown in the bottom panel of Fig.~\ref{fig:ex}. Note that in this example it is easy to simulate from the distribution $L$, and also from the distribution with density $f$ given at \eqref{eq:maxe61}. To simulate a random vector $(X_1,X_2)$ with distribution $L$, generate a random variable $U$ which is uniformly distributed on $[0,1]$ and set
\[
(X_1,X_2) = \begin{cases} (\sqrt{U},\sqrt{U}), & U \le 1/4,\\
((1 + \sqrt{4U - 1})/2,(1 - \sqrt{4U - 1})/2), & 1/4 < U < 1/2,\\
(U,U), & U \ge 1/2.\end{cases}
\]
A random vector $(X_1,X_2)$ with distribution given by the density $f$ at \eqref{eq:maxe61} is obtained by simulating independent random variables $U,V$ both uniformly distributed on $[0,1]$ and defining
\[
(X_1,X_2) = (1 - \sqrt{V}(1-U),U).
\]
\end{example}

\begin{example}
  Theorem~\ref{thr:order} establishes a relationship between the
  distribution of $(X_1,X_2)$ and the order statistics of a suitably
  chosen exchangeable pair $(V_1,V_2)$. Assume that $V_1$ and $V_2$ are
  independent. Then $\prob(V_1\leq x,V_2\leq x)=G(x)^2$. By
  Theorem~\ref{thr:order}, $F_1(x)=G(x)^2$, whence 
  \begin{displaymath}
    F_2(x)=2\sqrt{F_1(x)}-F_1(x), \quad x\in\R. 
  \end{displaymath}
  Then 
  \begin{displaymath}
    H(x)=2\sqrt{F_1(x)}(1-\sqrt{F_1(x)})
  \end{displaymath}
  is always unimodal. If $F_1$ is continuous the maximum attained at any lower quartile
  of $X_1$. The smallest possible Kendall's tau equals $-1/3$; it does
  not depend on $F_1$.  If we assume additionally that the support of $F_1$ is an interval, Theorem \ref{thm:SpearmanW} applies and the equation used to find $t(s)$ turns into
  $\sqrt{F_1(t(s))}=1-\sqrt{F_1(s)}$. Substituting this in
  \eqref{eq:sp-rho} yields that Spearman's rho equals $-1/2$ for all
  $F_1$.  If $F_1$ is absolutely continuous with density $f_1$ the maximum entropy is attained on the density
  $f(x_1,x_2)=2g(x_1)g(x_2)$, $x_1\geq x_2$, where
  $g(x)=f_1(x)/(2\sqrt{F_1(x)})$.
\end{example}

\begin{example}
  \label{ex:6.1}
  Let $X_1$ be uniform on $[0,1]$, and let $X_2=X_1^{1/\alpha}$ with
  $\alpha\in(0,1]$. Then $F_1(x)=x$, $F_2(x)=x^\alpha$, and 
  \begin{displaymath}
    F(x_1,x_2)=2\tilde{C}((x_1+x_1^\alpha)/2,(x_2+x_2^\alpha)/2)-x_2
  \end{displaymath}
  for $x_1\geq x_2$, see \eqref{eq:1}. In this case, \eqref{eq:5}
  yields that 
  \begin{displaymath}
    L(x_1,x_2) =
    \begin{cases}
      x_1, & x_1\leq x_2,\\
      x_2^\alpha-\min\{x_2^\alpha-x_2,x_1^\alpha-x_1\}, &
      \text{otherwise},
    \end{cases}
  \end{displaymath}
  The function $F_2-F_1$ is unimodal and attains its maximum at
  $r=\alpha^{(1-\alpha)^{-1}}$. The smallest Kendall's tau is
  \begin{displaymath}
    \tau=3-4 \;\E X_1^\alpha=\frac{3\alpha-1}{1+\alpha}.
  \end{displaymath}
  Note that $\tau=1$ if $\alpha=1$, $\tau=0$ if $\alpha=1/3$, and
  $\tau\to-1$ as $\alpha\downarrow0$. 

  We are not able to provide an explicit formula for Spearman's rho in
  terms of $\alpha$ but Fig.~\ref{fig:1} shows $\tau$ and $\rho$ as
  a function of $\alpha$.
\end{example}

\begin{figure}
\begin{center}
\includegraphics[width=0.5\textwidth]{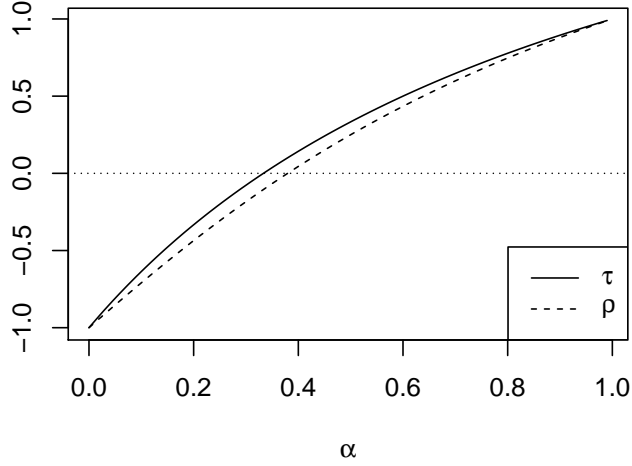}
\caption{Kendall's tau $\tau$ and Spearman's rho $\rho$ as functions
  of $\alpha \in (0,1]$ for the distributions in
  Example~\ref{ex:6.1}.\label{fig:1}}
\end{center}
\end{figure}

\begin{example}[Unimodal densities]
  \label{ex:6.2}
  Let $Z$ be a random variable with cdf $F$ and unimodal density $f$
  whose support is $\R$. Let $X_i=Z+\mu_i$, $i=1,2$,
  with $\mu_1\geq \mu_2$, hence $F_i = F(\cdot-\mu_i)$, $i\in\{1,2\}$. Then
  Theorem~\ref{thm:KendallW} yields for Kendall's tau of the
  distribution $L$ at \eqref{eq:5}
    \begin{displaymath}
    \tau=3-4\;\E F(Z+\mu_1-\mu_2)=4\int \check{F}(\mu_2-\mu_1-z)\diff F(z) -1,
  \end{displaymath}
  which is the convolution of $\check{F} (z):=1-F(-z)$ and $F$.  
  
  Let us additionally assume that $f$ is symmetric about its mode at
  zero. Then $\check{F} = F$, and the unimodal function $F_2-F_1$ has
  its maximum at $r=(\mu_1 + \mu_2)/2$. Hence, the function
  $t:(r,\infty) \to (-\infty,r)$ in Theorem \ref{thm:SpearmanW} is given by $t(s) = \mu_2 + \mu_1 -
  s$. If $F = \Phi$ is the standard Gaussian cdf, then $\Phi * \Phi(x)
  = \Phi(x/\sqrt{2})$, and, therefore, 
  \begin{displaymath}
    \tau = 4 \Phi((\mu_2-\mu_1)/\sqrt{2}) - 1.
  \end{displaymath}
  In particular, $\tau = 1$ if $\mu_1 = \mu_2$, $\tau = 0$ if $\mu_1 -
  \mu_2 = \sqrt{2}\Phi^{-1}(3/4) \approx 0.954$, and if $\mu_1 - \mu_2
  \to \infty$, then $\tau \to -1$.
  
  For Spearman's rho, we numerically computed the integrals in
  \eqref{eq:sp-rho} for $F = \Phi$. The values of $\tau$ and $\rho$ as
  functions of $\mu_1 - \mu_2 \ge 0$ are displayed in Fig.~\ref{fig:2}.
\end{example}

\begin{figure}
\begin{center}
\includegraphics[width=0.5\textwidth]{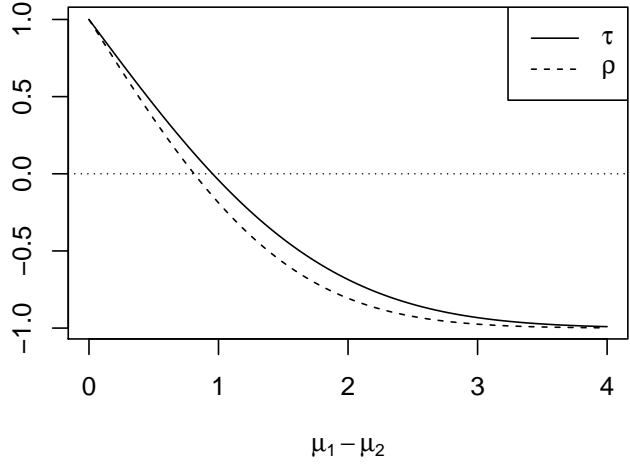}
\caption{Kendall's tau $\tau$ and Spearman's rho $\rho$ as functions
  of $\mu_1 - \mu_2 \in [0,4]$ for $F = \Phi$ in
  Example~\ref{ex:6.2}.\label{fig:2}} 
\end{center}
\end{figure}

\begin{example}[Exponential marginal distributions]
  \label{ex:6.3}
  Let $Z$ be an exponential random variable with cdf $F(x)=1-e^{-x}$,
  and let $X_i=Z/\lambda_i$, $i\in\{1,2\}$, with $\lambda_1\le
  \lambda_2$. Considering $\log(X_i)$, $i\in\{1,2\}$ shows that we
  are in the same setting as in Example~\ref{ex:6.2}
  because Kendall's tau and Spearman's rho are invariant under
  monotone transformations of the marginals and the stochastic
  ordering is preserved if we transform both marginals with the same
  increasing function.
  
  However, we can also compute $\tau$ and $\rho$ directly. Kendall's tau is given by
  \begin{displaymath}
    \tau = 4\int_0^1 (1-u)^{\lambda_2/\lambda_1} \diff u  - 1
    = \frac{3\lambda_1-\lambda_2}{\lambda_1 + \lambda_2}.
  \end{displaymath}
  Note that $\tau=1$ if $\lambda_1=\lambda_2$, $\tau = 0$ if
  $\lambda_1/\lambda_2 = 1/3$, and $\tau \to -1$ as
  $\lambda_1/\lambda_2 \downarrow 0$.

  The function
  \begin{displaymath}
    F_2(x)-F_1(x)=F(\lambda_2 x)-F(\lambda_1 x)=e^{-\lambda_1
      x}-e^{-\lambda_2 x}
  \end{displaymath}
  is unimodal on $[0,\infty)$ with maximum at 
  \begin{displaymath}
    r = (\log\lambda_2-\log\lambda_1)/(\lambda_2 - \lambda_1).
  \end{displaymath}
  Therefore, we can use \eqref{eq:sp-rho} to compute
  $\rho$. Considering the increasing transformation $\lambda_1 X_i$,
  $i\in\{1,2\}$, we see that $\tau$ and $\rho$ only depend on
  $\lambda_1/\lambda_2 \le 1$. Fig.~\ref{fig:3} provides plots of
  $\tau$ and $\rho$ as functions of $\lambda_1/\lambda_2 \in (0,1]$.
  \end{example}

\begin{figure}
\begin{center}
\includegraphics[width=0.5\textwidth]{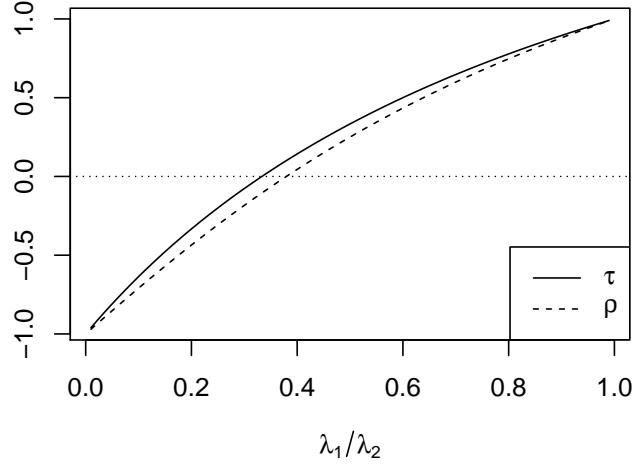}
\caption{Kendall's tau $\tau$ and Spearman's rho $\rho$ in terms of $\lambda_1/\lambda_2$ for exponentially distributed marginal distributions as in Example \ref{ex:6.3}.\label{fig:3}}
\end{center}
\end{figure}

\section{Generalizations}
\label{sec:generalizations}

A multivariate version of Theorem~\ref{thr:order} is the following.

\begin{theorem}
  \label{thr:orderM}
  A random vector $(X_1,\dots,X_n)$ with marginal cdfs $F_1,\dots,F_n$
  satisfies $\prob(X_1\geq \dots \geq X_n)=1$ if and only if $X_{i} =
  V_{(i)}$ where $V_{(1)} \ge \dots \ge V_{(n)}$ are the order
  statistics of a random vector $(V_1,\dots,V_n)$ with exchangeable
  components and such that for $j =1,\dots,n$
  \[
  \prob(V_1\leq x,\dots,V_j\leq x)
  = \frac{1}{\binom{n}{j}}\sum_{\ell=1}^{n-j+1}\binom{n-\ell}{j-1}F_\ell(x).
  \]
\end{theorem}
\begin{proof}[\textnormal{\textbf{Proof}}]
  Let $(V_1,\dots,V_n)$ be an exchangeable random vector that
  satisfies the above condition. By \cite[Proposition 4.4.1]{Lange2010}, we
  have that
  \begin{align*}
    \prob(V_{(i)} \le x)
    &= \sum_{j=n-i+1}^n (-1)^{j-(n-i+1)}\binom{j-1}{n-i}\binom{n}{j}
    \prob(V_1\leq x,\dots,V_{j}\leq x)\\
    &= \sum_{j=n-i+1}^n (-1)^{j-(n-i+1)}\binom{j-1}{n-i}
    \sum_{\ell=1}^{n-j+1}\binom{n-\ell}{j-1}F_\ell(x)\\
    &= \sum_{\ell=1}^i F_{\ell}(x) \binom{n-\ell}{n-i}
    \sum_{j=n-i+1}^{n-\ell+1}(-1)^{j-(n-i+1)}\binom{i-\ell}{n-\ell-j+1}\\ &= F_i(x).
  \end{align*}
  Conversely, if the vector $(V_1,\dots,V_n)$ is obtained as the random
  permutation of $(X_1,\dots,X_n)$, then it is exchangeable and the
  formula for $\prob(V_1\leq x,\dots,V_j\leq x)$ is essentially the  
  inversion of the first equality in the above equation.
\end{proof}

A variant of Theorem~\ref{thr:order} applies to random elements in
a lattice $E$ with partial order $\preceq$, and with $\vee$ being the
maximum and $\wedge$ being the minimum operation. Endow $E$ with the
$\sigma$-algebra generated by $\{y:\; y\preceq x\}$ for all $x\in
E$. Since these events form a $\pi$-system, the values $\prob(X\preceq
x)$, $x\in E$, uniquely determine the distribution of an $E$-valued
random element $X$.  

In this case, Theorem~\ref{thr:order} admits a direct generalization. 
Namely $X_1\preceq X_2$ a.s.~if and only if $X_1=V_1\vee V_2$ and
$X_2=V_1\wedge V_2$ for a pair $(V_1,V_2)$ of exchangeable random
elements in $E$ such that 
\begin{displaymath}
  \prob(V_i\preceq
  x)=\frac{1}{2}\Big(\prob(X_1\preceq x)+\prob(X_2\preceq x)\Big)
\end{displaymath}
and
$\prob(V_1\preceq x,V_2\preceq x)=\prob(X_1\preceq x)$.

\section*{Acknowledgement}

The problem of characterizing bivariate copulas with stochastically
ordered marginals was brought to the attention of the second author by
Nicholas Kiefer from the Economics Department at Cornell University. 

The authors are grateful to the referees for spotting mistakes in the
original version of this paper and suggesting numerous improvements.

\bibliographystyle{elsarticle-num.bst}
\bibliography{biblio}

\begin{thebibliography}{10}
\expandafter\ifx\csname url\endcsname\relax
  \def\url#1{\texttt{#1}}\fi
\expandafter\ifx\csname urlprefix\endcsname\relax\def\urlprefix{URL }\fi
\expandafter\ifx\csname href\endcsname\relax
  \def\href#1#2{#2} \def\path#1{#1}\fi

\bibitem{nav:spiz10}
J.~Navarro, F.~Spizzichino, On the relationships between copulas of order
  statistics and marginal distributions, Statist. Probab. Lett. 80 (2010)
  473--479.

\bibitem{leb:dut14}
R.~Lebrun, A.~Dutfoy, Copulas for order statistics with prescribed margins, J.
  Multivariate Anal. 128 (2014) 120--133.

\bibitem{diet:fuc:sch16}
M.~Dietz, S.~Fuchs, K.~D. Schmidt, On order statistics and their copulas,
  Statist. Probab. Lett. 117 (2016) 165--172.

\bibitem{DuranteMesiarETAL2005}
F.~Durante, R.~Mesiar, C.~Sempi, Copulas with given diagonal section: some new
  results, in: Proc. of the Joint 4th Conf. of the European Society for Fuzzy
  Logic and Technology and the 11th Rencontres Francophones sur la Logique
  Floue et ses Applications, 2005, pp. 931--936.

\bibitem{EmbrechtsHofert2013}
P.~Embrechts, M.~Hofert, A note on generalized inverses, Math. Method. Oper.
  Res. 77 (2013) 423--432.

\bibitem{SchweizerSklar1983}
B.~Schweizer, A.~Sklar, Probabilistic Metric Spaces, North-Holland Publishing
  Co., New York, 1983.

\bibitem{McNeilNeslehova2009}
A.~McNeil, J.~Ne\v{s}lehov\'a, Multivariate archimedean copulas, $d$-monotone
  functions and $l_1$-norm symmetric distributions, Ann. Stat. 37 (2009)
  3059--3097.

\bibitem{nel:ques:rod:04}
R.~B. Nelsen, J.~J. Quesada~Molina, J.~A. Rodr\'{\i}guez~Lallena, M.~\'{U}beda
  Flores, Best-possible bounds on sets of bivariate distribution functions, J.
  Multivariate Anal. 90 (2004) 348--358.

\bibitem{KlementKolesarova2005}
E.~P. Klement, A.~Koles\'arov\'a, Extensions to copulas and quasi-copulas as
  special 1-{L}ipschitz aggregation operators, Kybernetika 41 (2005) 329--348.

\bibitem{FredricksNelsen2002}
G.~A. Fredricks, R.~B. Nelsen, The {B}ertino family of copulas, in: C.~M.
  Cuadras, J.~Fortiana, J.~A. Rodr\'{i}guez-Lallena (Eds.), Distributions With
  Given Marginals and Statistical Modelling, Springer, Dordrecht, 2002.

\bibitem{rog99}
L.~C.~G. Rogers, Fastest coupling of random walks, J. London Math. Soc. (2) 60
  (1999) 630--640.

\bibitem{rac:rus98}
S.~T. Rachev, L.~R\"{u}schendorf, Mass Transportation Problems. {V}ol. {II},
  Springer-Verlag, New York, 1998.

\bibitem{Neslehova2007}
J.~Ne{\v{s}}lehov{\'a}, On rank correlation measures for non-continuous random
  variables, J. Multivariate Anal. 98 (2007) 544--567.

\bibitem{cov:thom06}
T.~M. Cover, J.~A. Thomas, Elements of Information Theory, 2nd Edition,
  Wiley-Interscience, Hoboken, NJ, 2006.

\bibitem{ButuceaDelmasETAL2018}
C.~Butucea, J.-F. Delmas, A.~Dutfoy, R.~Fischer, Maximum entropy distribution
  of order statistics with given marginals, Bernoulli 24 (2018) 115--155.

\bibitem{ButuceaDelmasETAL2015}
C.~Butucea, J.-F. Delmas, A.~Dutfoy, R.~Fischer, Maximum entropy copula with
  given diagonal section, J. Multivariate Anal. 137 (2015) 61--81.

\bibitem{Lange2010}
K.~Lange, Applied Probability, Springer, New York, 2010.

\end{thebibliography}

\end{document}